\documentclass[12pt, reqno]{amsart}
\usepackage{amssymb,latexsym,amsmath,amscd,amsfonts}
\usepackage{latexsym}
\usepackage[mathscr]{eucal}
\usepackage{bm}
\usepackage{hyperref}
\usepackage{cleveref}
\usepackage{enumerate}
\usepackage{xcolor}

9.5in \numberwithin{equation}{section}

\newcommand{\beg}{\begin{equation}}
    \newcommand{\eeg}{\end{equation}}
\newcommand{\ben}{\begin{eqnarray*}}
    \newcommand{\een}{\end{eqnarray*}}

\newtheorem{thm}{Theorem}[section]
\newtheorem{cor}[thm]{Corollary}

\newtheorem{prop}[thm]{Proposition}
\numberwithin{equation}{section} \theoremstyle{definition}
\newtheorem{defn}[thm]{Definition}
\newtheorem{rem}[thm]{Remark}

\newtheorem{qn}[thm]{Question}

\newcommand{\n}{\lVert}

\makeatletter \@namedef{subjclassname@2020}{\textup{2020}
Mathematics Subject Classification} \makeatother

\begin{document}

\title[Unitary bases]{Unitary orthonormal bases of finite dimensional inclusions}

\author[Bakshi and Bhat]{Keshab Chandra Bakshi and B V Rajarama Bhat }

 \address[ Department of Mathematics and Statistics,
         Indian Institute of Technology, Kanpur,
         Uttar Pradesh 208016, India]{}
   \email{keshab@iitk.ac.in, bakshi209@gmail.com}

   \address[Statistics and Mathematics Unit, Indian Statistical Institute, R V College Post, Bangalore - 560059, India]{}
    \email{bvrajaramabhat@gmail.com}




\maketitle

\begin{abstract}

We study unitary orthonormal bases in the sense of Pimsner and Popa
for inclusions $(\mathcal{B}\subseteq \mathcal{A}, E),$ where
$\mathcal{A}, \mathcal{B}$ are  finite dimensional von Neumann
algebras and $E$ is a conditional expectation map from $\mathcal{A}$
onto $\mathcal{B}$. It is shown that existence of such bases
requires that the associated inclusion matrix satisfies a spectral
condition forcing dimension vectors to be Perron-Frobenius
eigenvectors and the conditional expectation map  preserves the
Markov trace. Subject to these conditions, explicit unitary
orthonormal bases are constructed if either one of the algebras is
abelian or simple. They  generalize complex Hadamard matrices, Weyl
unitary bases, and a recent work of Crann et al which correspond to
the special cases of $\mathcal{A}$ being abelian, simple, and
general multi-matrix algebras respectively with $\mathcal{B}$ being
the algebra of complex numbers. For the first time $\mathcal{B}$ is
more general. As an application of these results it is shown that if
$(\mathcal{B}\subseteq \mathcal{A}, E),$ admits a unitary
orthonormal basis then the Connes-St{\o}rmer relative entropy
$H(\mathcal{A}_1|\mathcal{A})$ equals the logarithm of the square of
the norm of the inclusion matrix, where $\mathcal{A}_1$ denotes  the
 Jones basic construction of the inclusion. As a further
application, we prove the existence of unitary orthonormal bases for
a large class of depth 2 subfactors with abelian relative commutant.
\end{abstract}

\section{Introduction}

If we consider $\mathbb{C}^n$ as an algebra as well as a Hilbert
space with respect to standard inner product (normalized so that the
identity has norm 1), and look at unitary orthonormal bases we end
up with complex Hadamard matrices. Unlike the real case, they exist
in all dimensions. Complex Hadamard matrices such as finite Fourier
matrices have made their appearance in several different contexts
and there is extensive literature on the same. In a similar vein,
the full matrix algebra $M_n$, considered as a Hilbert space by
imposing the Hilbert-Schmidt inner product coming from the
normalized tracial state, also admits a variety of unitary
orthonormal bases. For instance it is well-known that Weyl unitaries
arising out of a projective unitary representation of the group
$\mathbb{Z}_n\times \mathbb{Z}_n$, yields one such  basis of this
space. These bases have found a wide variety of applications. Of
particular interest are the applications in quantum information
theory (\cite{Wat}). Some of them are known as nice error bases (See
\cite{Kn}, \cite{Kr}, \cite{BCF}). They are useful in quantum
teleportation and dense coding schemes as described by Werner in
\cite{Wer}.

Recently Crann, Kribs, and Pereira \cite{CKP} constructed unitary
orthonormal bases for arbitrary finite dimensional von Neumann
algebras on fixing suitable states on them. Their work showed us the
way to go further. In this article we generalize several of these
constructions of unitary orthonormal bases to the much wider setting
of Pimsner-Popa bases and demonstrate some applications. Except for
this introduction and the last section, this article deals with only
finite dimensional algebras. Whether we call them as von Neumann
algebras or $C^*$-algebras won't make a difference.  In the final
section we discuss applications of the theory developed to
subfactors. The applications to quantum information theory and other
fields remains to be explored. The beautiful combinatorial
structures and symmetries seen in these unitary bases make it
apparent that these constructions are likely to have further
applications. Already Conlon et al \cite{CCKL} have proposed some
uses of unitary orthonormal bases on general multi-matrix algebras
in quantum teleportation theory.

The theory of orthogonal bases developed by Pimsner and Popa
\cite{PP} considers bases  for von Neumann subalgebra-algebra pairs
with a chosen conditional expectation map. They are really
orthogonal bases for Hilbert von Neumann modules rather than Hilbert
spaces. Given a subfactor with finite Jones index (see \cite{Jo}),
Pimsner and Popa  proved that there always exists an orthonormal
basis corresponding to the unique trace-preserving conditional
expectation (see \cite{PP}). Subsequently, the concept of
Pimsner-Popa bases became an indispensable tool in subfactor theory.
It has been crucially used in the description of  planar algebras of
Jones which are regarded as  extremely important invariants for
subfactors. Motivated by  Pimsner-Popa bases, Watatani developed an
algebraic theory of index for a unital  inclusion of $C^*$-algebras
in \cite{W}. One can show that there always exists a Pimsner-Popa
basis for an inclusion of multi-matrix algebras with respect to any
faithful trace. Inclusions of finite dimensional von Neumann
algebras play a central role in the theory of subfactors. Indeed, a
certain grid of finite dimensional von Neumann-algebras defines the
so-called {\em `standard invariant'\/} of a subfactor which is a
complete invariant for a {\em `good class'\/} of subfactor with
finite Jones index. Furthermore, from the early days of subfactor
theory certain quadruples of finite dimensional von Neumann
algebras, called {\em commuting squares\/}  have been used to
construct a large class of hyperfinite subfactors. Thus, the study
of inclusion of finite dimensional von Neumann algebras is of
paramount importance in subfactor theory. Pimsner-Popa orthonormal
bases consisting of unitaries (for subfactors, Cartan MASAs, finite
dimensional inclusions etc.) arise naturally in von Neumann algebra
theory (see \cite{P,pop2}, for instance) and it has many
applications. One of the motivation of this paper is the following
open problem asked by Popa (see \cite{pop2}):
 \begin{qn} (Popa)
{\em Does there exist a unitary orthonormal basis for an integer
index irreducible subfactor?}
\end{qn}
Recently the first author and Gupta  have showed that any finite
index regular subfactor $N\subset M$ with $N^{\prime}\cap M$ is
either commutative or simple admits a unitary orthonormal basis (see
\cite{BG}). The crucial fact we used in the proof is the existence
of the unitary orthonormal basis for an inclusion
$\mathbb{C}\subseteq \mathcal{A}$, where $\mathcal{A}$ is either
simple or abelian. In \cite{BG}, it was conjectured that any regular
subfactor will have unitary orthonormal basis. This has been
answered positively in \cite{CKP} by proving that any inclusion of
finite dimensional $C^*$-algebras $\mathbb{C}\subseteq \mathcal{A}$
has unitary orthonormal basis with respect to the unique Markov
trace. It is natural to attempt to generalize this result for  more
general multi-matrix algebras. In this article we have two new
constructions of unitary orthonormal bases, one extending that of
\cite{CKP} in Section 5 and another extending Weyl unitaries in
Section 6.

 Consider a triple $(\mathcal{B} \subseteq \mathcal{A}, E)$ where
$\mathcal{A}, \mathcal{B}$ are finite dimensional von Neumann
algebras,  $\mathcal{B} \subset \mathcal{A}$ is a unital inclusion
and $E:\mathcal{A}\to \mathcal{B}$ is a conditional expectation map.
We call such triples as subalgebra systems. In  \textit{Section 2}
we  recall some basic facts and  the definition of orthonormal bases
for subalgebra systems.  In \textit{Section 3}, we derive necessary
conditions for the existence of unitary orthonormal basis. The
following theorem summarizes our results on this.
 \medskip

\textbf{Theorem A}: {\em  Let $(\mathcal{B} \subseteq  \mathcal{A},
E)$ be an inclusion of finite dimensional von Neumann algebras, with
$A\tilde{m}=\tilde{n},$ where $A$ is the inclusion matrix,
$\tilde{m}, \tilde{n}$ are dimension vectors as in
(\Cref{dimension}). Suppose $(\mathcal{A}, \mathcal{B}, E)$ admits a
unitary orthonormal basis with $d$-unitaries. Then
\begin{equation}\label{Spectral Condition}  A^t\tilde{n}= d\tilde{m}.\end{equation} Consequently,
$$A^tA\tilde{m}=d\tilde{m}; ~~AA^t\tilde{n}=d\tilde{n}.$$
In particular, $\tilde{m}, \tilde{n}$ are Perron-Frobenius
eigenvectors of $A^tA$ and $AA^t$ respectively with eigenvalue $d$.
Moreover $E$ is the unique conditional expectation preserving the
Markov trace with trace vector equal to the dimension vector
$\tilde{n}.$ }

\medskip

We call the conditions imposed on the inclusion matrix $A$ by this
theorem as the {\em spectral condition\/} and the condition on $E$
as the {\em trace condition\/.} We suspect that these conditions
could also be sufficient for the existence of unitary orthonormal
basis. Currently we are unable to prove this. However, we are able
to build unitary orthonormal bases for large classes of subalgebra
systems.

In \textit{Section 4}, we develop some basic notations and write
down various formulae for conditional expectation maps on
multi-matrix algebras. If a conditional expectation preserves the
standard trace then it is  a mixed unitary channel, and in such
cases we write them explicitly as convex combinations of a certain
number of unitary channels.  These concrete  descriptions of
conditional expectation maps help us to verify the orthogonality of
unitaries of various bases we construct in following sections.

In \textit{Section 5} we exhibit a unitary orthonormal basis for an
inclusion of multi-matrix algebras $\mathcal{B\subset A}$ with
either $\mathcal{B}$ or $\mathcal{A}$ abelian. This construction
makes heavy use of quasi-circulant matrices. It is inspired by the
recent work of Crann, Kribs, and Pereira (\cite{CKP}), who handled
the case of $\mathcal{B}=\mathbb{C}.$ The final answer can be seen
in Theorem \ref{First construction}. As a consequence we have the
following theorem.

\medskip

\textbf{Theorem B:} {\em Suppose the subalgebra system
$(\mathcal{B}\subseteq \mathcal{A},E)$ satisfies the spectral
condition and  $\mathcal{B}$
 is abelian. Then it has a unitary orthonormal basis.}

\medskip

It is a well-known fact that the subalgebra  system $(\mathbb{C}\subseteq M_n,
\frac{1}{n}\mbox{tr})$  admits a unitary orthonormal basis. A standard basis called Weyl unitaries consists of a
family of the form $\{ V^jU^k:0\leq j,k\leq (n-1)\}$ where $V$ is a
cyclic shift and $U$ is a diagonal unitary with roots of unity on
the diagonal. Generalizing this construction, in \textit{Section 6},
  we exhibit a unitary basis where we
replace $\mathbb{C}$ by a general finite dimensional abelian algebra
and $M_n$ is replaced by a direct sum of copies of $M_n$ with
inclusion satisfying the spectral conditions necessitated by Theorem
A.

 In \textit{Section 7}, we describe various methods of getting
 unitary orthonormal bases for new subalgebra systems constructed
 out of subalgebra systems already having this property. The methods
 we have are concatenation, taking tensor products or direct sums and the
 basic construction of Jones. These devices are good enough to prove
 the following.

\medskip

\textbf{Theorem C:} {\em Let $(\mathcal{B}\subseteq \mathcal{A}, E)$
be a subalgebra systems of  finite dimensional $C^*$-algebras.
\begin{itemize}
\item[(i)]  If $\mathcal{B}=M_{m}$ with $E$ preserving the Markov trace then
$(\mathcal{B}\subseteq \mathcal{A}, E)$ has unitary orthonormal basis.

\item [(ii)]  If $\mathcal{A}=M_{n}$ with $E$
preserving the unique trace on $M_{n}$ then  $(\mathcal{B}\subseteq
\mathcal{A}, E)$ has unitary orthonormal basis.
\end{itemize}}

 \medskip

 In the concluding \textit{Section 8}, we provide
two applications. As a first application, we show that if $N\subset
M$ is a depth 2 subfactor with $N^{\prime}\cap M$ abelian and
$N^{\prime}\cap M\subset N^{\prime}\cap M_1$ is superextremal (see
\cite{L}),  then the trace preserving conditional expectation from $M$ onto $N$ has
unitary orthonormal basis. To achieve this we crucially use the
spectral condition as in Theorem A.

For the  second and final application, we consider an inclusion of
finite dimensional von Neumann algebras $\mathcal{B\subset A}$ with
an inclusion matrix $A$ and suppose $\mathcal{A}_1$ is the Jones
basic construction (see \cite{Jo})  corresponding to the Markov
trace $\tau$. If the unique $\tau$-preserving conditional
expectation has a unitary orthonormal basis, then we prove that
$H(\mathcal{A}_1|\mathcal{A})=\ln {\lVert A \rVert}^2, $ where $H$
denotes the Connes-St{\o}rmer relative entropy (see \cite{CS}). To
prove this we have  used the spectral condition and Theorem B.

\section{Preliminaries}

Our basic set up will be a triple $(\mathcal{B}\subseteq
\mathcal{A}, E)$ where $\mathcal{A}, \mathcal{B}$ are finite
dimensional von Neumann algebras with $\mathcal{B} $ being a unital
 subalgebra of $\mathcal{A}$ and $E:\mathcal{A}\to \mathcal{B}$ is
a conditional expectation map, that is, $E$ is a unital completely
positive map satisfying $E(Y)=Y, ~~\forall Y\in B$. We will call
such a triple as a sub-algebra system. It is well-known that any
such map $E$ has the {\em `bimodule property'\/}: $$E(YXZ)=YE(X)Z,
~~\forall Y,Z\in \mathcal{B}, X\in \mathcal{A}.$$ Given a faithful
trace $\varphi $ on $\mathcal{A}$ there exists a unique conditional
expectation map $E_{\varphi }$ on $\mathcal{A}$ preserving $\varphi
.$ For further information on conditional expectation maps we refer
to \cite{U}.

\begin{defn} Consider a subalgebra system  $(\mathcal{B}\subseteq \mathcal{A}, E)$ as above. A family $\{W_j:0\leq
j\leq (d-1)\}$ (for some $d\in \mathbb{N}$) of elements of
$\mathcal{A}$ is called a \textit{ (right) Pimsner-Popa basis} for
$(\mathcal{B}\subseteq \mathcal{A}, E)$ if
\begin{equation}\label{right basis} X= \sum _{j=0}^{d-1}W_jE(W_j^*X),
~~\forall ~X\in \mathcal{A}.\end{equation} It is said to be an
\textit{ unitary orthonormal basis} if $W_j$ is a unitary for every
$j$, and
\begin{equation}\label{orthonormality} E(W_j^*W_k)= \delta
_{jk},~~\forall ~0\leq j,k\leq (d-1).\end{equation}
\end{defn}
 The word `right' in the nomenclature is suppressed as we will not
 be dealing with any other kind of bases. It is known that all finite dimensional subalgebra systems admit
Pimsner-Popa bases. However, as we shall see, existence of unitary
orthonormal bases is not always guaranteed. So it is useful to have
the following definition \cite{S}.

\begin{defn} A subalgebra system
$(\mathcal{B}\subseteq \mathcal{A}, E)$ is said to have
\textit{$U$-property} if it admits a unitary orthonormal basis.
\end{defn}
Now we introduce our standard setting and  develop some notation.
Throughout we will take
\begin{equation}\label{decompositionA}
\mathcal{A}= M_{n_0}\oplus M_{n_1}\oplus \cdots \oplus M_{n_{s-1}};
\end{equation}
\begin{equation}\label{decompositionB}
\mathcal{B}= M_{m_0}\oplus M_{m_1}\oplus \cdots \oplus
M_{m_{r-1}}.\end{equation}
 The {\em dimension vectors\/} of $\mathcal{A}$,
$\mathcal{B}$ are given by
\begin{equation}\label{dimension}
\tilde{n} = \left(\begin{array}{c}
n_0\\n_1\\\vdots
\\n_{s-1}\end{array}\right),
~~\tilde{m}= \left(\begin{array}{c} m_0\\m_1\\\vdots
\\m_{r-1}\end{array}\right)
\end{equation}
 respectively. The {\em inclusion matrix\/}  of $\mathcal{B}$ in
$\mathcal{A}$ is given by an $s\times r$ matrix:
$$A=[a_{ij}]_{0\leq i\leq (s-1); 0\leq j\leq (r-1).} $$
where the algebra $M_{m_j}$  of $\mathcal{B}$ appears $a_{ij}$ times
in the algebra $M_{n_i}$ of $\mathcal{A}.$ Any inclusion of finite
dimensional von Neumann algebras $\mathcal{B}\subseteq \mathcal{A}$
is completely determined up to isomorphism by this triple  $(A,
\tilde{m}, \tilde{n})$ of one inclusion matrix and two dimension
vectors. Note that by dimension counting,
\begin{equation}
\sum _{j=0}^{r-1} a_{ij}m_j= n_i, ~~\forall ~0\leq i\leq (s-1)
\end{equation}
or in matrix  notation:
\begin{equation}
\tilde{n}= A\tilde{m}.\end{equation}

Any faithful tracial state $\varphi $ on $\mathcal{A}$ is given by
 \begin{equation}\label{tracial state}
 \varphi(\oplus _{i=0}^{s-1}X_i)=\frac{1}{\sum _{i=0}^{s-1}p_in_i} \sum _{i=0}^{s-1}p_i
 ~\mbox{trace}(X_i),~~\oplus _{i=0}^{s-1}X_i\in
 \mathcal{A},\end{equation}
for some scalars $p_i>0$ with $\sum _{i=0}^{s-1}p_i=1.$ We call the
column vector $\tilde{p}=(p_0, p_1, \ldots , p_{s-1})^t $ as a {\em
trace vector\/} of $\varphi ,$ which is uniquely determined from
$\varphi $ up to multiplication by a positive scalar.  Any such
state $\varphi $ is said to be a {\em Markov state} if
\begin{equation}\label{Markov}
AA^t\tilde{p}= \|A\|^2\tilde{p}.
\end{equation}
This means that $\tilde{p}$ is a Perron-Frobenius eigenvector for
the non-negative matrix $AA^t.$

From Crann et. al. \cite{CKP} we know that the sub-algebra system
$(\mathbb{C}\subseteq \mathcal{A}, E)$  admits a unitary basis, if
the conditional expectation $E$ is given by
\begin{equation}\label{state}
E(X)= \varphi (X)I, ~~\forall ~X\in \mathcal{A},\end{equation}
 with $\varphi $
being the tracial state with trace vector $\tilde{n}$, that is
\begin{equation}\label{the state}
\varphi (\oplus X_i)= \frac{1}{\sum _{i=0}^{s-1}n_i^2} \sum
_{i=0}^{s-1}n_i~\mbox {trace}~ (X_i), ~~X_i \in M_{n_i}, 1\leq i\leq
(s-1).\end{equation} Note that in this case, the inclusion matrix is
a column matrix, $A=[n_0, \ldots , n_{s-1}]^t$ and $\varphi $ is a
Markov state.

\begin{rem}
 If for a subalgebra system $(\mathcal{B\subset A},E)$ with inclusion
matrix $A$, if $d:={\lVert A \rVert}^2$ is an integer and there are
$d$ many unitaries satisfying the orthonormality condition
(\ref{orthonormality}), then they form a basis, that is, (\ref{right
basis}) is automatic (See \cite{Bak} [Theorem 2.2]).
\end{rem}

\section{Markov state preservation and the spectral condition}

In this section we will see various properties of inclusions of
finite dimensional algebras having unitary orthonormal basis. In
other words we are deriving some necessary conditions for subalgebra
systems to have $U$-property.

We wish to show that $E$ preserves the tracial state with trace
vector equal to the dimension vector $\tilde{n}$ of $\mathcal{A}.$
Later we will see that this is actually the Markov state for the
inclusion, when the inclusion  has $U$-property. First we work out a
special case. Here we observe that the construction of unitary
orthonormal basis in \cite{CKP} required a Markov state is not a
coincidence. It is inescapable.

\begin{prop}\label{trace}
Suppose $(\mathbb{C}\subseteq \mathcal{A}, E)$ has $U$-property
where
$$\mathcal{A}= M_{n_0}\oplus M_{n_1}\oplus \cdots \oplus
M_{n_{s-1}}.$$ Then $E$ is given by \Cref{state} and \Cref{the state}.
\end{prop}

\begin{proof}
 Suppose $\{W_0, W_1, \ldots , W_{d-1}\}$ is a unitary o.n.b. for
 the triple $(\mathbb{C}\subseteq \mathcal{A}, E).$  Clearly we  have $d=n_0^2+n_1^2+\cdots +n_{s-1}^2$, as that is the dimension of
 $\mathcal{A}.$  By general theory,
 $$E(X) = \mbox{trace}(\rho X).1, ~~\forall ~X\in \mathcal{A}$$
 for some density matrix,
 $$\rho = \oplus _{i=0}^{s-1}\rho _i$$
 where $\rho _i$ is a positive matrix in $M_{n_i}, ~1\leq i\leq
 (s-1).$ Let
 $$\rho _i=\sum _{a=0}^{n_i-1}p_{ia}|v_{ia}\rangle \langle v_{ia}|,
 $$
 be a spectral decomposition of $\rho _i$ for $0\leq i\leq (s-1),$
 so that $\{ v_{ia}: 0\leq i\leq (s-1); 0\leq a\leq (n_i-1)\}$ is an
 orthonormal basis of $\mathbb{C}^{n_0}\oplus \mathbb{C}^{n_1}\oplus
 \cdots \oplus \mathbb{C}^{n_{s-1}},$ and $\sum _{i,a}p_{ia}=1.$  Let
 $$W_j=W_{j0}\oplus W_{j1}\oplus \cdots \oplus W_{j(s-1)}$$
 be the decomposition of $W_j$ in $\mathcal{A}=
 \mathcal{M}_{n_0}\oplus \mathcal{M}_{n_1}\oplus \cdots \oplus
 \mathcal{M}_{n_{s-1}}$ for $0\leq j\leq (d-1).$ From the orthonormality
 of the basis we have,
 \begin{eqnarray*}
 \delta _{jk}&=& \mbox{trace}~(\rho W_j^*W_k)\\
&=& \sum _{i=0}^{s-1} ~\mbox{trace}~(\rho _iW_{ji}^*W_{ki})\\
&=& \sum _{i=0}^{s-1}\sum
_{a=0}^{n_i-1}p_{ia}~\mbox{trace}~(|v_{ia}\rangle \langle
v_{ia}|W_{ji}^*W_{ki})\\
&=& \sum _{i=0}^{s-1}\sum _{a=0}^{n_i-1}p_{ia}\langle W_{ji}v_{ia},
W_{ki}v_{ia}\rangle .
\end{eqnarray*}
That is,  \begin{equation}\label{orthogonality} \delta _{jk}= \sum
_{i=0}^{s-1}\sum _{a=0}^{n_i-1}\langle \sqrt{p_{ia}}W_{ji}v_{ia},
\sqrt{p_{ia}}W_{ki}v_{ia}\rangle .
\end{equation}
For fixed $j$, $\sqrt{p_{ia}}W_{ji}v_{ia}$ is a vector in
$\mathbb{C}^{n_i}.$ Therefore, taking
$$w_j:= ( \sqrt{p_{ia}}W_{ji}v_{ia})_{ 0\leq i\leq (s-1); 0\leq a\leq
(n_i-1)}$$ we have a vector in $\mathbb{C}^{n_0^2+n_1^2+\cdots
+n_{s-1}^2}= \mathbb{C}^{d}$ and by \Cref{orthogonality} these
vectors are orthonormal as we vary $j$. Taking them as column
vectors, we get a $d\times d$ unitary matrix
$$W:=[ w_0, w_1, \ldots , w_{d-1}].$$
Therefore, the $l^2$-norm of every row of $W$ is one. On the other
hand, the sum of $l^2$-norms of the first $n_0$-rows of $W$ is given
by
$$\sum _{j=0}^{d-1}\|
\sqrt{p_{00}}W_{j0}v_{00}\|^2 = \sum _{j=0}^{d-1}p_{00}.1=
p_{00}.d.$$ This gives $n_0=dp_{00}$ or $p_{00}=\frac{n_0}{d}.$ This
we got by taking $(i,a)=(0,0)$ and adding over $j$. Similarly, for
any fixed $(i,a)$,
$$n_i= \sum _{j=0}^{d-1}\|\sqrt{p_{ia}}W_{ji}v_{ia}\|^2= \sum
_{j=0}^{d-1}p_{ia}.1= d.p_{ia}.$$ This shows that
$$p_{ia}=\frac{n_i}{d}.$$
In other words,
$$\rho _i= \sum _{a=0}^{n_i-1}p_{ia}|v_{ia}\rangle \langle v_{ia}|=
\sum _{a=0}^{n_i-1}\frac{n_i}{d}|v_{ia}\rangle \langle v_{ia}|=
\frac{n_i}{d}.I_i$$ where $I_i$ is the identity in the appropriate
space. Hence,
$$ \rho = \oplus _i(\frac{n_i}{d}).I_i.$$ So
$\varphi (\oplus _iX_i)= \frac{1}{d}\sum
_{i=0}^{s-1}n_i~\mbox{trace}~(X_i).$
\end{proof}

Now we look at the general case. We need the following tool, which
says that unitary orthonormal bases may be obtained by concatenation.

\begin{prop} \label{con}Suppose $(\mathcal{A}_1\subseteq \mathcal{A}_0, E_0)$ and $(\mathcal{A}_2\subseteq \mathcal{A}_1, E_1)$ are two
subalgebra systems having $U$-property. Then
$(\mathcal{A}_2\subseteq \mathcal{A}_0, E_1E_0)$ has $U$-property.
Indeed, if $\{V_j: j\in J\}$ are $\{W_k: k\in K\}$ are unitary
orthonormal basis for $(\mathcal{A}_1\subseteq \mathcal{A}_0, E_0)$
and $(\mathcal{A}_2\subseteq \mathcal{A}_1, E_1)$ respectively, then
$\{V_jW_k: j\in J, k\in K\}$ is an unitary orthonormal basis for
$(\mathcal{A}_2\subseteq \mathcal{A}_0, E_1E_0)$.
\end{prop}

\begin{proof} This follows by direct computation using the bi-module property of the conditional expectation $E_0.$
\end{proof}

\begin{thm}\label{Markov trace}
Let $(\mathcal{B}\subseteq \mathcal{A}, E)$ be an inclusion of
finite dimensional von Neumann algebras having $U$-property. Suppose
$\mathcal{A}= M_{n_0}\oplus M_{n_1}\oplus \cdots \oplus
M_{n_{s-1}}$. Then $E$  is the unique conditional expectation
preserving the tracial state $\varphi $ given by \Cref{the state}.
\end{thm}

\begin{proof}
By  \cite{CKP} the subalgebra system $( \mathbb{C}\subseteq
\mathcal{B}, E_1)$ has $U$-property for some conditional expectation
map $E_1$. Then by \Cref{con} $( \mathbb{C}\subseteq \mathcal{A},
E_1E)$ has $U$-property. Thanks to \Cref{trace}, $E_1E$ preserves
the state $\varphi $. This means that
\begin{equation}\label{2.1}
E_1E(X)= \varphi (X)I,~~\forall X\in \mathcal{A}.
\end{equation}
Replacing $X$ by $E(X)$ in this equation, we get
$$E_1E (E(X))= \varphi (E(X))I, ~~\forall X\in \mathcal{A}.$$
Since $E^2=E$, this yields
 \begin{equation}\label{2.2}
 E_1E(X)=
\varphi (E(X))I, ~~\forall X\in \mathcal{A}.
\end{equation}
From  \Cref{2.1} and \Cref{2.2},
$$\varphi (X)= \varphi (E(X)), ~~\forall X\in \mathcal{A}.$$ \end{proof}

In the previous result we can not say that the state $\varphi $ is
the unique tracial state preserved by the conditional expectation.
For instance, when $\mathcal{B}=\mathcal{A}$, the conditional
expectation map is the identity map and it preserves every state.
However, this is not the case in most non-trivial situations.

\begin{cor}
For $n\geq 1$, suppose $\varphi $ is a faithful state on
$M_n(\mathbb{C})$, such that the Hilbert space $M_n(\mathbb{C})$
with inner product
$$\langle X, Y\rangle = \varphi(X^*Y), ~~X,Y\in M_n(\mathbb{C})$$
admits a unitary orthonormal basis. Then $\varphi $ must be the
normlized standard trace on $M_n(\mathbb{C}).$
\end{cor}
\begin{proof}
This is just the special case of the previous theorem with
$\mathcal{A}=M_n(\mathbb{C}), \mathcal{B}=\mathbb{C}$ and $E$ given
by $E(X)= \varphi(X)I, ~~X\in \mathcal{A}.$

\end{proof}

 Now we show a very important spectral property  for the inclusion matrix of a subalgebra system admitting
 a unitary orthonormal basis.  This has several consequences.

 \begin{thm}\label{spectral}
Let $(\mathcal{B}\subseteq \mathcal{A}, E)$ be a subalgebra system,
with
$$A\tilde{m}=\tilde{n},$$ where $A$ is the inclusion matrix, and
$\tilde{m}, \tilde{n}$ are dimension vectors as in
(\Cref{dimension}). Suppose $(\mathcal{B}\subseteq \mathcal{A}, E)$
admits a unitary orthonormal basis with $d$-unitaries.
 Then
\begin{equation}
 A^t\tilde{n}= d\tilde{m},
\end{equation}
 Consequently,
$$A^tA\tilde{m}=d\tilde{m}; ~~AA^t\tilde{n}=d\tilde{n}.$$
In particular, $\tilde{m}, \tilde{n}$ are Perron-Frobenius
eigenvectors of $A^tA$ and $AA^t$ respectively with eigenvalue $d$.

 \end{thm}

\begin{proof}
 Let $P_i$ be the projection onto the $i$-th summand in the
decomposition $\mathcal{A}= \oplus _{i=0}^{s-1}M_{n_i},$ and
similarly let $Q_j$ be the projection onto the $j$-th summand in the
decomposition  $\mathcal{B}= \oplus _{j=0}^{r-1} M_{m_j}$ for $0\leq
i\leq (s-1); 0\leq j\leq (r-1).$ Note that $P_i, Q_j$ are minimal
central projections of $\mathcal{A}, \mathcal{B}$, respectively. We
will compute the dimension of the vector space,
$$\mathcal{A}Q_j =\{XQ_j: X\in \mathcal{A}\}$$
for fixed $j$, with $0\leq j\leq (r-1)$ in two different ways.

Recall that  $\mathcal{A}P_i$ gives us the $i$-th component of
$\mathcal{A}$, in which the $j$-th component $M_{m_j}$ of
$\mathcal{B}$ appears $a_{ij}$ times. Therefore $\mathcal{A}P_iQ_j$
consists of rectangular matrices of size $n_i\times a_{ij}m_j$ in
$M_{n_i}$. Hence
$$\mbox{dim}~(\mathcal{A}P_iQ_j)= n_ia_{ij}m_j.$$
As we have a direct sum decomposition,
\begin{equation}\label{nA}
\mbox{dim}~(\mathcal{A}Q_j)= \sum _{i=0}^{s-1}n_ia_{ij}m_j.
\end{equation}
Let $\{W_0, W_1, \ldots , W_{d-1}\}$ be an unitary orthonormal basis
for the subalgebra system $(\mathcal{B}\subseteq  \mathcal{A}, E).$
Let $\{Y_0, Y_1, \ldots , Y_{m_j^2-1}\}$ be a basis for
$\mathcal{B}Q_j.$ We claim that $\{W_kY_l: 0\leq k\leq (d-1); 0\leq
l\leq (m_j^2-1)\}$ is a basis for $\mathcal{A}Q_j$.

If $X\in \mathcal{A}$, then $X= \sum _{k=0}^{d-1}W_kE(W_k^*X)$ So,
\begin{equation}\label{AQ_j} XQ_j= \sum _{k=0}^{d-1}W_kE(W_k^*X)Q_j,
\end{equation}
Note that $E(W_k^*X)$ is in $\mathcal{B}$. This shows that
$\{W_kY_l\}$ spans $\mathcal{A}Q_j.$ Now if $\sum c_{kl}W_kY_l=0$
for some scalars $c_{kl}.$ For any fixed $a$, by orthonormality of
$W_k$'s,
$$0 = E(W_a^*(\sum _{k,l}c_{kl}W_kY_l))= \sum
_{k,l}c_{kl}E(W_a^*W_k)Y_l=\sum _lc_{al}Y_l,
$$ which implies $c_{al}=0$. So $\{W_kY_l\}$ are linearly independent and the dimension of $\mathcal{A}Q_j$ is $d.m_j^2$. Combining with
equation (\ref{nA}), we get $\sum _{i=0}^{s-1}n_ia_{ij}m_j=dm_j^2$
or equivalently $A^t\tilde{n}= d\tilde{m}.$
\end{proof}.
\begin{cor} Suppose a subalgebra system $(\mathcal{B}\subseteq
\mathcal{A}, E)$ with inclusion matrix $A$ has $U$-property. Then
the square of the norm of $A$ is an integer and equals to the number
of elements in the unitary basis.
\end{cor}
\begin{proof}
This is clear from the previous theorem as Perron-Frobenius
eigenvalue of a non-negative matrix is its spectral radius and the
spectral radius of $A^tA$ is equal to $\|A\|^2.$

\end{proof}

\begin{cor}
Suppose a finite dimensional inclusion $(\mathcal{B}\subseteq
\mathcal{A}, E)$ has $U$-property. Then $E$ is the unique
conditional expectation preserving the Markov trace.
\end{cor}

\begin{proof}  Follows from the definition of Markov trace and the previous result.
\end{proof}

Combining Theorem \ref{Markov trace} and Theorem \ref{spectral}, we
have Theorem A.  We also note the following interesting quadratic
relationship. This shouldn't be very surprising as we have similar
relations for group algebras (See comments in the beginning of
Section 7).

\begin{cor}\label{quadratic}
Let $(\mathcal{B}\subseteq \mathcal{A}, E)$ be a subalgebra system
with inclusion matrix $A$ and dimension vectors $\tilde{n},
\tilde{m}.$ Suppose it has $U$-property with $d$-unitaries in the
basis. Then the vector space dimensions of $\mathcal{A},\mathcal{B}$
are related by,
\begin{equation}
\sum _in_i^2=d \sum _jm_j^2
\end{equation}
\end{cor}

\begin{proof}
With notation as before,
$$\langle \tilde{m}, A^tA\tilde{m}\rangle =\langle \tilde{m},
A^t\tilde{n}\rangle = d\langle \tilde{m}, \tilde{m}\rangle= b\sum
_jm_j^2.$$ Also,
$$\langle \tilde{m}, A^tA\tilde{m}\rangle= \langle A\tilde{m},
A\tilde{m}\rangle= \langle \tilde{n}, \tilde{n}\rangle = \sum
_in_i^2.$$

\end{proof}

\begin{rem}
Following the arguments in the proof of Theorem \ref{spectral},
where we assume only a Pimsner-Popa basis, which may not be
orthogonal, all the steps go through, except verification of the
linear independence. This leads to the inequality:
$$A^t\tilde{n}\leq d\tilde{m}.$$
\end{rem}

\section{The conditional expectation}

In this section we introduce some  notation and describe how the
conditional expectation map looks like for finite dimensional
inclusions. We also write down canonical trace preserving
conditional expectation maps as mixed unitaries.

Here and in subsequent sections we need to deal with roots of unity
of different orders. It is very convenient to have the following
notation. Define $\epsilon :\mathbb{R}\to \mathbb{T}$ by
\begin{equation}\label{exponential function}\epsilon (x) = e^{2\pi ix}, ~~x\in \mathbb{R}.\end{equation}
 We make repeated use of the following elementary observations.

\begin{rem}\label{exponential}  Let $\epsilon $ be the exponential function defined
as above. Then (i) $\epsilon (x+y)=\epsilon(x). \epsilon(y),
~~\forall x,y\in \mathbb{R};$ (ii) $\epsilon (-x)= \overline
{\epsilon (x)}, ~~\forall x\in \mathbb{R};$ (iii) $\epsilon (x)=1$
iff $x\in \mathbb{Z}.$ (iv) For $k\in \mathbb{N}$, $x\in
\mathbb{R},$
$$\sum _{j=0}^{k-1}\epsilon (jx)= \left\{ \begin{array}{cl}
0& ~~\mbox{if}~ kx\in \mathbb{Z}, ~x\notin \mathbb{Z}\\
k& ~~\mbox{if}~x\in \mathbb{Z} \end{array}\right. $$
\end{rem}

Consider two finite dimensional von Neumann algebras: $$\mathcal{A}=
\oplus _{i=0}^{s-1}M_{n_i}, ~~~\mathcal{B}=\oplus
_{j=0}^{r-1}M_{m_j}$$ with $\mathcal{B}\subseteq \mathcal{A}$,
having an inclusion matrix $A=[a_{ij}]_{1\leq i \leq (s-1); 1\leq
j\leq (r-1)}.$ Here it would be assumed that for every $j$, there
exists some $i$, such that $a_{ij}\neq 0$. Otherwise, a component of
$\mathcal{B}$ would be missing in $\mathcal{A}$ and we don't have unital inclusion $\mathcal{B}\subseteq \mathcal{A}.$

 Fix any such faithful tracial state $\varphi $ on
$\mathcal{A}$ as in \ref{tracial state}.  We know that there exists
a unique conditional expectation from $\mathcal{A}$ to $\mathcal{B}$
which preserves $\varphi .$ In this section we wish to compute this
conditional expectation. This will be useful to verify orthogonality
of unitaries of various bases we are going to construct in
subsequent sections. We also show that if $\varphi $ preserves the
standard trace on $\mathcal{A}$ (that is when $p_i$'s are all
equal), the conditional expectation is a mixed unitary channel and
can be written down explicitly in such a form.

The algebra $\mathcal{A}$ acts naturally on a Hilbert space
$\mathcal{H}$  of dimension $n_0\oplus \cdots  \oplus n_{s-1}.$ It
is convenient to choose a basis which encodes the inclusion of
$\mathcal{B}.$ We do this by choosing an orthonormal basis:
$$\{ u_{ijkl}: 0\leq i\leq (s-1); 0\leq j\leq (r-1); 0\leq k\leq
(a_{ij}-1); 0\leq l\leq (m_j-1)\},$$ for $\mathcal{H}$  with
following understanding:

\begin{enumerate}
\item  For fixed $i$, $|u_{ijkl}\rangle \langle u_{ij'k'l'}|$, with
$j,k,l,j',k',l'$ varying,  form the matrix units of the $i$-th
summand $M_{n_i}$ of $\mathcal{A}.$

\item  For fixed $i,j, k$, the matrix units of $k$-th copy
of $M_{m_j}$ of $\mathcal{B}$ in $M_{n_i}$, are given by
$|u_{ijkl}\rangle \langle u_{ijkl'}|$ with $0\leq l,l' \leq
(m_j-1).$

\item  For fixed $j$, the matrix units of $\mathcal{B}$ are given
by
$$\sum _{i=0}^{ (s-1)}\sum _{k=0}^{(a_{ij}-1)}|u_{ijkl}\rangle \langle
u_{ijkl'}|.$$
\end{enumerate}

Note that if $a_{ij}=0$,  $M_{m_j}$ does not appear in $M_{n_i}$. In
such a case, there is no vector of the form $u_{ijkl}$.

Putting together different diagonal blocks created by $\mathcal{B}$
we get a new algebra,
$$\mathcal{C}:= \oplus _{i=0}^{s-1}\oplus _{j=0}^{r-1}\oplus
_{k=0}^{(a_{ij}-1)}M_{m_j},$$ where the matrix units of $M_{m_j}$
for fixed $i,j,k$ are as in (ii).  Clearly this is an intermediate
von Neumann algebra: $\mathcal{B}\subseteq \mathcal{C}\subseteq
\mathcal{A}.$

First we consider a conditional expectation of $\mathcal{A}$ onto
$\mathcal{C}$. We will call it $E_1$. It is easy to write it down.
For fixed $i,j,k,$ let $P_{ijk}$ be the projection on to the
subspace spanned by  $\{ u_{ijkl}:0\leq l\leq (m_j-1)\}.$ For any
$X\in \mathcal{A}$, take
$$E_1(X) = \sum _{i,j,k}P_{ijk}XP_{ijk}.$$
It is simply pinching of $X$, to the diagonal blocks. Clearly $E_1$
is a unital completely positive map. The diagonal entries of
$E_1(X)$ (in the chosen basis) are same as that of $X$. Hence $E_1$
preserves {\em every} tracial state on $\mathcal{A}.$ Furthermore,
$$E_1(|u_{ijkl}\rangle \langle u_{ijkl'}|)= |u_{ijkl}\rangle \langle
u_{ijkl'}|.$$ Hence $E_1(Y)=Y,$ for $Y\in \mathcal{C}.$ This shows
that $E_1$ is a conditional expectation from $\mathcal{A}$ to
$\mathcal{C}.$

Now we wish to compute the conditional expectation from
$\mathcal{C}$ to $\mathcal{B}$, which preserves $\varphi .$ Recall
that $\mathcal{C}$ contains only certain diagonal blocks. The
conditional expectation $E_2$ is got by taking weighted averages of
these blocks and the weights depend upon the state $\varphi $ we
want to preserve.  Indeed take \begin{equation}\label{q}q_{ij}=
\frac{p_i}{\sum _{x=0}^{s-1}p_xa_{xj}}.\end{equation}
 It is well
defined as $p_x$'s are strictly positive and  for every $j$,
$a_{xj}\neq 0$ for some $x$, and so we have  $\sum
_{x=0}^{s-1}p_xa_{xj}\neq 0.$  Note that $\sum
_{i=0}^{s-1}a_{ij}q_{ij}=1.$

Consider $Y= \oplus _{ijk} Y_{ijk}$ in $\mathcal{C}= \oplus
_{ijk}M_{m_j}$. We take, $E_2(Y)= \oplus _{ijk}Z_{ijk}$ where
$$Z_{ijk} := \sum _{v=0}^{s-1}\sum _{w=0}^{a_{vj}-1}q_{vj}Y_{vjw}.$$
Note that $Z_{ijk}$ does not depend upon $i$ or $k$. Hence $\oplus
_{ijk}Z_{ijk}\in \mathcal{B}$ and $E_2$ is a CP map. Also,
\begin{eqnarray*}
\varphi(E_2(Y))&= & \frac{1}{N}
\sum_{ijk}p_i~\mbox{trace}~(Z_{ijk})\\
&=& \frac{1}{N}\sum _{ijk}p_i
(\sum_{v=0}^{s-1}\sum_{w=0}^{a_{vj}-1}q_{vj}~\mbox{trace}(Y_{vjw}))\\
&=& \frac{1}{N}\sum _{ij}
\sum_{v=0}^{s-1}\sum_{w=0}^{a_{vj}-1}p_ia_{ij}q_{vj}~\mbox{trace}(Y_{vjw})\\
&=& \frac{1}{N}\sum _{j}
\sum_{v=0}^{s-1}\sum_{w=0}^{a_{vj}-1}p_v~\mbox{trace}(Y_{vjw})\\
&=& \varphi (Y).
\end{eqnarray*}
If $Y\in \mathcal{B}$, then $Y_{ijk}$ does not depend upon $i,k$ and
so we get
$$ \sum _{v=0}^{s-1}\sum _{w=0}^{a_{vj}-1}q_{vj}Y_{vjw}=
Y_j.\sum _{v=0}^{s-1}a_{vj}q_{vj}=Y_j.$$ Therefore $E_2$ is a UCP
map, fixing elements of $\mathcal{B}.$

Now take $E:=E_2\circ E_1: \mathcal{A}\to \mathcal{B}.$ Being a
composition of conditional expectation maps, it is a conditional
expectation. Since both $E_1, E_2$ preserve $\varphi $, $E$ also
preserves $\varphi .$

\subsection{Mixed unitary channels}

We wish to write down the conditional expectation map $E$ as a mixed
unitary channel, that is, we want to have a family of unitaries
$\{U_k: k\in K\}$ such that the map $E$ is a convex combination of
maps of the form: $X\mapsto U_kXU_k^*.$ Clearly any such map must
preserve the standard trace. So in this subsection we assume that
the state $\varphi $ is given by
$$\varphi (\oplus X_i)= \frac{1}{\sum _xn_x}~\sum
_{i=0}^{s-1}\mbox{trace}(X_i).$$

We consider the lexicographic order on $\{(i,j): 0\leq i\leq (s-1);
0\leq j\leq (r-1)\}$ by setting $(i_1, j_1)<(i_2, j_2)$ if either
$i_1<i_2$ or $i_1=i_2$ and  $j_1<j_2.$  Take $T_j=\sum
_{i=0}^{s-1}a_{ij}$ and $T=\sum T_j$. In other words, $T_j$ is the
number of copies of $M_{m_j}$ of $\mathcal{B}$ in $\mathcal{A}$ and
$T$ is the total number of such blocks. Then the formula for weights
(\ref{q}) simplifies to, $q_{ij}= \frac{1}{\sum
_{x=0}^{s-1}1.a_{xj}}= \frac{1}{T_j}$ for every $i,j.$

 Look at the map $E_1$ defined before. To express it as
a mixed unitary channel, define $K$ on $\mathcal{H}$ by
$$K(u_{ijkl})= \epsilon (\frac{1}{T}[\sum _{(i_1, j_1)<(i,j)}a_{i_1j_1}+k]) u_{ijkl} .$$
It means that $K$ is a diagonal operator, which acts as scalar
$\epsilon (\frac{1}{T}{\sum _{(i_1, j_1)<(i,j)}a_{i_1j_1}+k})$ on
the $(i,j,k)$-th block of $\mathcal{C}.$ Take
$$F_1(X)= \frac{1}{T} \sum _{x=0}^{T-1}K^xX(K^x)^*, X\in
\mathcal{A}.$$ We claim that $F_1$ is same as $E_1.$
 Taking powers of $K$, $$K^x(u_{ijkl})= \epsilon (\frac{x}{T}[ \sum _{(i_1,
 j_1)<(i,j)}a_{i_1j_1}+k]) u_{ijkl}.$$ Hence,
$$ F_1( |u_{ijkl}\rangle \langle u_{i'j'k'l'}| )
 = \sum _{x=0}^{T-1}\epsilon (\frac{x}{T}[\sum _{(i_1,j_1)<
 (i,j)}a_{i_1j_1}+k -\sum _{(i_2, j_2)<(i', j')}a_{i_2j_2}-k'])
 |u_{ijkl}\rangle \langle u_{i'j'k'l'}|.
 $$
From Remark  \ref{exponential}, concerning basic properties of the
exponential function $\epsilon $, the last sum is zero, unless
$$\sum _{(i_1,j_1)<
 (i,j)}a_{i_1j_1}+k =\sum _{(i_2, j_2)<(i', j')}a_{i_2j_2}+k'.$$
This happens iff $(i,j,k)= (i', j',k')$ as $ \sum _{(i_1,j_1)<
 (i,j)}a_{i_1j_1}+k$ on varying $i,j,k$ is just an enumeration of all the integers from
 $0$ to $T-1$. Consequently,
$$F_1(|u_{ijkl}\rangle \langle u_{i'j'k'l'}|)= \left\{
\begin{array}{cl}
0& ~\mbox{if}~ (i,j,k)\neq (i',j',k');\\
|u_{ijkl}\rangle \langle u_{i'j'k'l'}|& ~\mbox{Otherwise}.
\end{array}\right.$$ Comparing with the action of $E_1$, it is clear
that $F_1=E_1$.

Now we look at $E_2$. For fixed $j\in \{0, 1, \ldots , (r-1)\}$,
take $\mathcal{G}_j=\{(i,k): 0\leq i\leq (s-1); 0\leq k\leq
a_{ij}-1\}.$  We order the elements of $\mathcal{G}_j$ also in
lexicographic order. Note that $\mathcal{G}_j$ has $T_j$ many
elements and $T_j\neq 0.$ We may name the elements of
$\mathcal{G}_j$ as $\{g_0, g_1, \ldots , g_{T_j-1}\}$ (The
dependence on $j$ is suppressed here in notation). Let $\sigma
:\mathcal{G}_j\to \mathcal{G}_j$ be the cyclic permutation of
$\mathcal{G}_j$, defined by $\sigma (g_y)= g_{y+1},$ for $0\leq y<
T_j-1$ and $\sigma (g_{T_j-1})=g_0\}.$

For fixed $j_1$, define $L_{j_1}$ by setting
$$L_{j_1}u_{ijkl}= \left\{ \begin{array}{cl}
u_{ijkl}&~\mbox{if}~ j\neq j_1;\\
u_{i'jk'l}&~\mbox{if}~j=j_1; \sigma ((i,k))=(i',
k').\end{array}\right.$$ In other words, $L_{j_1}$ acts as identity
on basis vectors $u_{ijkl}$ with $j\neq j_1$ and if $j=j_1$ it
cyclically permutes $(i,k)$ in $\mathcal{G}_{j_1}.$ We consider the
powers of $L_{j_1}$, $L_{j_1}^x$ for $0\leq x\leq T_{j_1}-1$ and set
$$G_{j_1}(Y)= \frac{1}{T_{j_1}}\sum _{x=0}^{T_{j_1}-1}
L_{j_1}^xY(L_{j_1}^x)^*,~~Y\in \mathcal{C}.$$ Clearly $G_{j_1}$ is a
mixed unitary channel on $\mathcal{C}.$ It averages all the blocks
in $\mathcal{C}$ corresponding to $M_{m_{j_1}}$ of $\mathcal{B}.$ In
particular, $$ G_{j_1}(|u_{ij_1kl}\rangle \langle u_{ij_1kl'}|)=
\frac{1}{T_{j_1}}\sum _{x=0}^{T_{j_1}-1}L_{j_1}^x|u_{ij_1kl}\rangle
\langle
u_{ij_1kl'}|(L_{j_1}^x)^*\\
= \frac{1}{T_{j_1}}\sum _{(i_1,k_1)\in
\mathcal{G}_{j_1}}|u_{i_1j_1k_1l}\rangle \langle u_{i_1j_1k_1l'}|
$$
and if $j\neq j_1$,
$$G_{j_1}(|u_{ijkl}\rangle \langle u_{ijkl'}|) =
\frac{1}{T_{j_1}}\sum _{x=0}^{T_{j_1}-1}L_{j_1}^x|u_{ijkl}\rangle
\langle u_{ijkl'}|(L_{j_1}^x)^* = |u_{ijkl}\rangle \langle
u_{ijkl'}|.
$$
We may carry this averaging process with respect to every $j$. Then
it is clear the map $Y\mapsto G_0\circ G_1\circ \cdots \circ
G_{r-1}(Y), ~~Y\in \mathcal{C}$ is same as the conditional
expectation map $E_2$ defined earlier.

\begin{thm}
Let $(\mathcal{B}\subseteq \mathcal{A}, E)$ be a subalgebra system,
where $E$ is the conditional expectation preserving the standard
trace. Then $E$ is a mixed unitary channel: $E(X)=$
$$ \frac{1}{T_0T_1\cdots T_{s-1}.T}\sum
_{x_0=0}^{T_0-1}\sum _{x_1=0}^{T_1-1}\cdots \sum
_{x_{r-1}=0}^{T_{r-1}-1}\sum _{y=0}^{T-1}L_0^{x_0}L_1^{x_1}\cdots
L_{{r-1}}^{x_{r-1}}K^yX(K^y)^*(L_{{r-1}}^{x_{r-1}})^*\cdots
(L_1^{x_1})^*(L_0^{x_0})^*.$$
\end{thm}
\begin{proof}
As  both $E_1, E_2$ are mixed unitary channels, $E=E_2\circ E_1$ is
also a mixed unitary by composition.

\end{proof}

It is to be noted that in this description of the conditional
expectation map, the unitaries in general may not be elements of
$\mathcal{A}$. They are from the ambient space $M_{n_0+n_1+\cdots
+n_{s-1}}.$

\section{The abelian case}

In this Section we construct a unitary orthonormal basis for a
subalgebra system $(\mathcal{B}\subseteq \mathcal{A}, E)$ where
either $\mathcal{A}$ or $\mathcal{B}$ is abelian. This is a
generalization of the construction in \cite{CKP}, where the special
case of  $\mathcal{B}=\mathbb{C}$ was considered. To begin with note
that if $\mathcal{A}$ is abelian then so is $\mathcal{B}.$ Therefore
it suffices to consider the case where $\mathcal{B}$ is abelian.

 At first we introduce a class of matrices. The unitaries in our constructions of
orthonormal bases are invariably in this class. Recall the
exponential function $\epsilon $ defined in equation
(\ref{exponential function}).  For any natural number $n$, the
$n\times n$, finite Fourier matrix $F_n$ is the unitary matrix given
by
$$F_n= \frac{1}{\sqrt{n}}[\epsilon (\frac{jk}{n})]_{0\leq j,k\leq
(n-1)}.$$ Denote the $n\times n$ diagonal matrix with diagonal
entries $a_0, \ldots , a_{n-1}$ by $D(a_0, \ldots ,a_{n-1}).$
Similarly for $n$-complex numbers $b_0, b_1, \ldots , b_{n-1}$ the
{\em circulant matrix\/}  $C(b_0, \ldots , b_{n-1})$ is defined as:
$$C(b_0, \ldots , b_{n-1})= F_n^*D(b_0, \ldots , b_{n-1})F_n.$$
Note that we are parametrizing circulant matrices by their
eigenvalues and not by their matrix entries. This simplifies our
computations.

\begin{defn} A matrix of the form $D_1CD_2$, where $D_1, D_2$ are
diagonal and $C$ is circulant is called a {\em quasi-circulant\/}
matrix.
\end{defn}

We collect together a few basic facts about quasi-circulant matrices
in the form a proposition for easy reference. It is to be noted that
the notion is basis dependent.

\begin{prop}
Let $a_0, \ldots , a_{n-1}, b_0, \ldots , b_{n-1}, c_0, \ldots ,
c_{n-1}$ be complex numbers.
\begin{enumerate}
\item If $a_j$'s and $b_j$'s are of modulus $1$ then $D(a_0, \ldots , a_{n-1})$ and $C(b_0, \ldots ,
b_{n-1})$ are unitary matrices with eigenvalues $a_0, \ldots ,
a_{n-1}$ and $b_0, \ldots , b_{n-1}$ respectively.

\item For any natural number $r$,
$$ (D(a_0, \ldots , a_{n-1}))^r= D(a_0^r, \ldots ,
a_{n-1}^r);~~~~ (C(b_0, \ldots , b_{n-1}))^r=C(b_0^r, \ldots ,
b_{n-1}^r).$$

\item $C(b_0, \ldots , b_{n-1})=\frac{1}{n}[\sum _{y=0}^{n-1}b_y\epsilon (\frac{y(k-j)}{n}) ]_{0\leq j,k\leq (n-1)}.$
 In particular the diagonal entries of $C(b_0, \ldots , b_{n-1})$ are
all equal and equal to $\frac{1}{n}\sum _yb_y.$

\item $C(b_0, \ldots , b_{n-1})D(a_0, \ldots , a_{n-1})= \frac{1}{n}[\sum _{y=0}^{n-1}b_y\epsilon (\frac{y(k-j)}{n})a_k]_{0\leq j,k\leq (n-1)},$
and its trace is given by $\frac{1}{n} (\sum _yb_y)(\sum _xa_x).$

\item $D(a_0, \ldots , a_{n-1})C(b_0, \ldots , b_{n-1})D(c_0, \ldots
, c_{n-1})=  \frac{1}{n}[a_j\sum _{y=0}^{n-1}b_y\epsilon
(\frac{y(k-j)}{n})c_k]_{0\leq j,k\leq (n-1)}$ and its trace is given
by $\frac{1}{n}(\sum _yb_y).(\sum _xa_xc_x).$

\item The diagonal of $D(a_0, \ldots , a_{n-1})C(b_0, \ldots , b_{n-1})D(c_0, \ldots
, c_{n-1})$ is same as  that of $C(b_0, \ldots , b_{n-1})D(a_0,
\ldots , a_{n-1})D(c_0, \ldots , c_{n-1}).$

\item The circulant matrix $C(1, \frac{1}{n}, \frac{2}{n}, \ldots ,
\frac{n-1}{n})$ is the permutation matrix which permutes the
standard basis $\{e_0, e_1, \ldots , e_{n-1}\}$ cyclically, mapping
$e_j$ to $e_{j+1}$ (with addition modulo $n$) for $0\leq j\leq n-1.$
\end{enumerate}
\end{prop}

\begin{proof}
The results (1) to (6) follow by direct computations and (7) follows
from (3), as the $jk$'th entry of  $C(1, \frac{1}{n}, \frac{2}{n},
\ldots , \frac{n-1}{n})$ is given by $$\frac{1}{n}\sum _{y=0}^{n-1}
\epsilon(\frac{y}{n})\epsilon(\frac{y(k-j)}{n})=\frac{1}{n} \sum
_{y=0}^{n-1}\epsilon (\frac{y(k+1-j)}{n})=\left\{
\begin{array}{cl} 1& ~\mbox{if}~j=k+1;\\
0&~\mbox{otherwise.}\end{array}\right.$$
\end{proof}

 Now let us take $\mathcal{B}$ as abelian. Then our basic setup
 reduces to
$$\mathcal{A}= \oplus _{i=0}^{s-1}M_{n_i},~~\mathcal{B}=\mathbb{C}^{r},$$
and the dimension vector $\tilde{m}$ has all entries equal to 1.
Consequently the inclusion matrix  $A=[a_{ij}]_{s\times r}$
satisfies $n_i= \sum _{j=0}^{r-1}a_{ij}$ for all $i$ and the
spectral condition implies $\sum _{i=0}^{s-1}n_ia_{ij}=d $ for all
$j$, for some natural number $d.$
 The conditional expectation $E$
should preserve the state $\varphi $ on $\mathcal{A}$, where
$$\varphi (\oplus_i X_i)=\frac{1}{\sum _in_i^2}\sum
_in_i~\mbox{trace}(X_i), ~~\oplus _iX_i \in \mathcal{A}.$$ Further,
$\mathcal{B}=\mathbb{C}^r$ implies that the orthonormal basis
$\{u_{ijkl}\}$ has $l$ identically equal to $0$. Therefore, we may
suppress $l$ from the notation and take the basis as
$$\{u_{ijk}: 0\leq i\leq (s-1); 0\leq j\leq (r-1); 0\leq k\leq
a_{ij}-1\}.$$

Let $P_0, \ldots , P_{s-1}$ be the minimal central projections of
$\mathcal{A}$, that is,  $P_i$ is the projection on to the span of
$\{u_{ijk}: 0\leq j\leq r-1; 0\leq k\leq a_{ij}-1\}.$ The
$(i,j,k)$th -diagonal entry of $X:=\oplus _i X_i$ is given by
$x_{ijk}:=\langle u_{ijk}, Xu_{ijk}\rangle .$ Now the conditional
expectation $E$ is given by
\begin{eqnarray*} E(\oplus _iX_i)&=& \sum _{j=0}^{r-1}\frac{(\sum
_{i=0}^{s-1}\sum
_{k=0}^{a_{ij}-1}n_ix_{ijk})}{\sum _in_ia_{ij}}P_j\\
&=& \frac{1}{d} \sum _{j=0}^{r-1}(\sum _{i=0}^{s-1}\sum_
{k=0}^{a_{ij}-1}n_ix_{ijk})P_j.\end{eqnarray*} In particular,
$E(X)=0$ if and only if for every $j$, \begin{equation}\label{sum}
\sum _{i=0}^{s-1}\sum_ {k=0}^{a_{ij}-1}n_ix_{ijk} =0.
\end{equation}

 Now we define
two unitary matrices  $U= \oplus _{i=0}^{s-1}U_i,~~V=\oplus
_{i=0}^{s-1}V_i, ~~\mbox{in}~\oplus _{i=0}^{s-1}M_{n_i}.$ The
spectrum of these matrices will be various powers of the $d$-th root
of unity. For notational convenience we take  $b:=\frac{1}{d}.$

Let $V_i$ be the circulant matrix $$V_i=C(1, \epsilon({b}),
\epsilon({2}{b}), \ldots , \epsilon({(n_i-1)}{b})).$$ The matrix
$U_i$ is chosen to be a direct sum of diagonal matrices (and hence
diagonal) with respect to the basis chosen above: $$U_i= \oplus
_{j=0}^{r-1} \epsilon ({\sum _{x=0}^{i-1}n_xa_{xj}}{b})D(1,
\epsilon({n_i}{b}), \epsilon({2n_i}{b}), \ldots ,
\epsilon({n_i(a_{ij}-1)}{b}))$$ In other words, the $(i,j,k)$-th
diagonal entry of $U$ is given by $\epsilon( \sum
_{x=0}^{i-1}n_xa_{xj}b+kn_i{b}).$

\begin{thm}\label{abelain unitary}  With notation as above,
$$\{V^tU^t:0\leq t\leq (d-1)\}$$ is a unitary orthonormal basis for
$(\mathbb{C}^r\subseteq \mathcal{A}, E)$.
\end{thm}

\begin{proof}
From basic properties of quasi-circulant mentioned above,
$$U^t= \oplus _{i=0}^{s-1}U_i^t, ~~V^t= \oplus _{i=0}^{s-1}V_i^t,$$
where \begin{eqnarray*} V_i^t &=& C(1, \epsilon({t}{b}),
\epsilon({2t}{b}), \ldots
,  \epsilon ({(n_i-1)t}{b}))\\
U_i^t&=& \oplus _{j=0}^{r-1}\epsilon(   {\sum
_{x=0}^{i-1}n_xa_{xj}}{tb} ) D(1, \epsilon({n_i}{tb}),
\epsilon({2n_i}{tb}), \ldots ,
\epsilon({(a_{ij}-1)n_i}{tb})).\end{eqnarray*}

Now $V_i^t$ being a circulant matrix it is possible to compute its
entries. Note that it is an $n_i\times n_i$ matrix and it acts on
the span of $\{ u_{ijk}: 0\leq j\leq (r-1), 0\leq k\leq
(a_{ij}-1)\}.$ For fixed $i$, arranging these vectors using the
lexicographic order of subscripts, the location of $u_{ijk}$ is
$\sum _{v=0}^{j-1}a_{iv}+k$. Hence the entries of $V^t$ are given
by:
$$ (V^t)_{ijk, i'j'k'}= \delta _{i,i'}. \frac{1}{n_i}\sum
_{y=0}^{n_i-1}\epsilon(ytb)\epsilon(\frac{y(\sum
_{v'=0}^{j'-1}a_{iv'}+k'-\sum _{v=0}^{j-1}a_{iv}-k)}{n_i}).$$ Where
as $U^t$ is a diagonal matrix with diagonal entries being,
$$(U^t)_{ijk,ijk}=\epsilon (\sum _{x=0}^{i-1}n_xa_{xj}tb+
kn_itb).$$ Therefore the entries of $W(t):=V^tU^t$ are given by
\begin{equation}\label{The big formula}
W_{ijk, i'j'k'}(t)= \frac{\delta _{ii'}}{n_i}\sum
_{y=0}^{n_i-1}\epsilon (ytb+\frac{y(\sum
_{v'=0}^{j-1}a_{iv'}+k'-\sum _{v=0}^{j-1}a_{iv}-k)}{n_i} +
{kn_itb+\sum _{x=0}^{i-1}n_xa_{xj}}{tb} ).
\end{equation}

In particular the  $(i,j,k)$-th diagonal entry of $W(t)=V^tU^t$ is
given by
\begin{equation}\label{diagonal}W_{ijk}(t):=[\frac{1}{n_i}{ \sum _{y=0}^{n_i-1}\epsilon (ytb)}
] \epsilon({kn_itb+\sum _{x=0}^{i-1}n_xa_{xj}}{tb}).\end{equation}
For fixed $j$,%

\begin{eqnarray*} \sum_{i=0}^{s-1}\sum _{k=0}^{a_{ij}-1}n_iW_{ijk}(t) &=& \sum
_{i=0}^{s-1}\sum _{k=0}^{a_{ij}-1}[\sum _{y=0}^{n_i-1}\epsilon
(ytb)].\epsilon({kn_itb+\sum
_{x=0}^{i-1}n_xa_{xj}}{tb})\\
&=& \sum _{i=0}^{s-1} \epsilon(\sum _{x=0}^{i-1}n_xa_{xj}tb)\left(
\sum _{k=0}^{a_{ij}-1}\epsilon(kn_itb)\left(\sum
_{y=0}^{n_i-1}\epsilon(ytb)\right)\right)\\
&=& \sum _{i=0}^{s-1} \epsilon(\sum _{x=0}^{i-1}n_xa_{xj}tb)\left(
\sum _{z=0}^{n_ia_{ij}-1}\epsilon(ztb)\right)\\
&=& \sum _{v=0}^{d-1}\epsilon({v}{tb})\\
&=& \sum _{v=0}^{d-1}\epsilon(\frac{vt}{d}) ,\end{eqnarray*} which
is zero if $t\neq 0.$

From this computation using equation (\ref{sum}),  we see that
$E(W(t))=0$ for $t\neq 0$. Now for any $v$,
$E((W(v))^*W(t)=E((V^vU^v)^*V^tU^t) = E((U^{-v})^*V^{t-v}U^t).$ As
$U$ is diagonal and $V$ is circulant, the  diagonal of
$U^{-v}V^{t-v}U^t$ is same as that of $V^{t-v}U^{t-v}$. Hence the
conditional expectation $E((U^{-v})^*V^{t-v}U^t)=
E(V^{t-v}U^{t-v})=E(W(t-v))=0$ for $t\neq v$, proving claimed
orthogonality.

\end{proof}

In the previous theorem to prove the orthogonality of the basis
elements it sufficed to know only their diagonal entries. The
theorem includes the construction of \cite{CKP} as a special case.
For reader's convenience we write the formulae explicitly for this
situation. Consider the inclusion $(\mathbb{C}, \mathcal{A}, E)$
where $\mathcal{A}=\oplus _{i=0}^{s-1}M_{n_i}$ and $E(\oplus X_i)=
\sum _{i=0}^{s-1}n_i~\mbox{trace}(X_i)$. The inclusion matrix is
same as the dimension vector $\tilde{n}.$  Now the basis consists of
$\{u_{ik}: 0\leq i\leq (s-1); 0\leq k\leq (n_i-1)\}$ and $d=\sum
_{i=0}^{s-1}n_i^2.$ As before, taking $b=\frac{1}{d}$, the unitary
operators generating the basis are given by
$$U^t= \oplus _{i=0}^{s-1}U_i^t, ~~V^t= \oplus _{i=0}^{s-1}V_i^t,$$
where \begin{eqnarray*} V_i^t &=& C(1, \epsilon({t}{b}),
\epsilon({2t}{b}), \ldots
,  \epsilon ({(n_i-1)t}{b}))\\
U_i^t&=& \epsilon(   {\sum _{x=0}^{i-1}n_x^2}{tb} ) D(1,
\epsilon({n_i}{tb}), \epsilon({2n_i}{tb}), \ldots ,
\epsilon({(n_i-1)n_i}{tb})).\end{eqnarray*} The unitary basis is
given by $W(t)=\oplus _{i=0}^{s-1} W_i(t)=\oplus
_{i=0}^{(s-1)}V_i^tU_i^t, 0\leq t\leq d-1$, where
\begin{equation}\label{The special case}
(W_i(t))_{k,k'}= \frac{1}{n_i}\sum _{y=0}^{n_i-1}\epsilon
(ytb+\frac{y(k'-k)}{n_i} + {kn_itb+\sum _{x=0}^{i-1}n_x^2}{tb} ).
\end{equation}

\section{Generalizing Weyl unitaries}

The subalgebra  system $(\mathbb{C}\subseteq M_n,
\frac{1}{n}\mbox{tr})$  admits a unitary orthonormal basis is
well-known. A standard basis called Weyl unitaries consists of a
family of the form $\{ V^jU^k:0\leq j,k\leq (n-1)\}$ where $V$ is a
cyclic shift and $U$ is a diagonal unitary with roots of unity on
the diagonal.

Generalizing this construction we exhibit a unitary basis where we
replace $\mathbb{C}$ by a general finite dimensional abelian algebra
and $M_n$ is replaced by a direct sum of copies of $M_n$ with
inclusion satisfying the spectral condition.

Consider subalgebra system $(\mathcal{B}, \mathcal{A}, E)$ where
\begin{eqnarray*} \mathcal{B} &=& \mathbb{C}\oplus \cdots \oplus \mathbb{C} ( r
~\mbox{times}) = \mathbb{C}^r,\\
\mathcal{A}&=& M_n\oplus \cdots \oplus M_n (s ~\mbox{times})\\
\end{eqnarray*} with inclusion  matrix
$A=[a_{ij}]_{r\times s}.$ The dimension vector of $\mathcal{B}$ has
all entries equal to $1$ and the dimension vector of $\mathcal{A}$
has all entries equal to $n$. Now the spectral condition implies
that $\sum _{i=0}^{s-1}a_{ij}=q$ for some fixed natural number $q$,
independent of $j$. The number of basis vectors will be $d=nq$. The
quadratic condition (\ref{quadratic}) implies $rq=ns.$ The Markov
trace on $\mathcal{A}$ is given by
$$\varphi(\oplus _{i=0}^{s-1}X_i) =\frac{1}{ns}\sum
_{i=0}^{s-1}~\mbox{trace}(X_i), ~~\oplus _iX_i \in \mathcal{A}.$$ In
our standard set up, as $\mathcal{B}=\mathbb{C}^r,$ the basis for
the Hilbert space reduces to
$$\{ u_{ijk}: 0\leq i\leq s-1; 0\leq j\leq r-1; 0\leq k\leq
a_{ij}-1\}.$$  Denoting the diagonal entries by, $x_{ijk}=\langle
u_{ijk}X_iu_{ijk}\rangle ,$ the conditional expectation is given by
$E(\oplus _iX_i)= \oplus Y_i,$ where $Y_i$ is a diagonal matrix with
$Y_{ijk,ijk}= \sum _{i=0}^{s-1}\sum _{k=0}^{a_{ij}-1}x_{ijk}.$

First we consider two unitary operators. Take,
$$V= \oplus _{i=0}^{s-1}V_i,$$
where $V_i= C(1, \epsilon(\frac{1}{n}), \epsilon(\frac{2}{n}),
\ldots , \epsilon (\frac{n-1}{n})\}.$ Note that $V_i$ cyclically
permutes the basis vectors. Consequently diagonals of powers $V_i^v$
are all equal to zero for $0<v<n.$

Define $U$ by
$$Uu_{ijk}= \epsilon(\frac{1}{q}(\sum
_{x=0}^{i-1}a_{xj}+k))u_{ijk}.$$ So $U$ is a diagonal unitary and
for any $t\in \mathbb{Z},$
$$U^tu_{ijk}= \epsilon(\frac{t}{q}(\sum
_{x=0}^{i-1}a_{xj}+k))u_{ijk}.$$

\begin{thm}\label{First construction}
With notation as above, $\{ V^vU^t: 0\leq v\leq n-1, 0\leq t\leq
q-1\}$ is a unitary orthonormal basis for the subalgebra system
$(\mathbb{C}^r \subseteq \mathcal{A}, E).$
\end{thm}
\begin{proof}
Consider the diagonal entries of $V^vU^t$. Since $U^t$ is diagonal
and $V^v$ has all the diagonal entries equal to $0$ we get
$$\langle u_{ijk}, V^vU^tu_{ijk}\rangle =\left\{ \begin{array}{cl}
0 & ~\mbox{if}~v\neq 0;\\
 \epsilon(\frac{t}{q}(\sum
_{x=0}^{i-1}a_{xj}+k))& ~\mbox{if}~v=0.\end{array}\right.$$ Now, for
fixed $j$, for $t\neq 0$,
$$\sum _{i=0}^{s-1}\sum _{k=0}^{a_{ij}-1}\epsilon(\frac{t}{q}(\sum
_{x=0}^{i-1}a_{xj}+k))=\sum
_{y=0}^{q-1}\epsilon(\frac{t}{q}(y))=0,$$ as $(\sum
_{x=0}^{i-1}a_{xj}+k)$ on varying  $i,k$ is just an enumeration of
all natural numbers from $0$ to $q-1.$ Hence
$$\sum _{i=0}^{s-1}\sum _{k=0}^{a_{ij}-1}\langle u_{ijk},
V^vU^tu_{ijk}\rangle =0.$$ It follows that $E(V^vU^t)=0$ if
$(v,t)\neq (0,0).$

Now for any $v_1, v_2, t_1, t_2$, $(V^{v_1}U^{t_1})^*V^{v_2}U^{t_2}=
U^{-t_1}V^{v_2-v_1}U^{t_2}$. As $U$ is diagonal and $V$ is
circulant, the diagonal of $U^{-t_1}V^{v_2-v_1}U^{t_2}$ is same as
that of  $V^{v_2-v_1}U^{t_2-t_1}$. Therefore if $(v_1, t_1)\neq
(v_2, t_2)$,
$$E(V^{v_1}U^{t_1})^*V^{v_2}U^{t_2}=E(V^{v_2-v_1}U^{t_2-t_1})=0.$$
\end{proof}

We remark that unlike Weyl unitaries  the orthonormal basis
constructed here need not be a  projective unitary representation.

\section{New bases from the old and the full matrix algebra}

  It is well-known that if $\mathcal{A}$ is
the complex group algebra of a finite group $G$ and $\mathcal{B}$ is
the group algebra of a subgroup $H$ of $G$, then with suitable
conditional expectation map $E$, $(\mathcal{B}\subseteq \mathcal{A},
E)$ admits a unitary orthonormal basis (See Example 1.2.3 of
\cite{W}). It is obvious that not all subalgebra systems appear as
subalgebra systems of group algebras.  Characterizing such
inclusions seems to be a difficult problem. The question as to which
multi-matrix algebras are group algebras is itself open and is known
as Brauer's Problem 1 (\cite{Mor}). The crossed product construction
(See Example 2.3 of \cite{CCKL}) is another standard method of
constructing subalgebra systems with unitary orthonormal bases.

 In the last two sections we have explicitly
constructed unitary orthonormal bases for inclusions where the
subalgebra under consideration is abelian. Here we describe a few
methods of constructing unitary orthonormal bases using the known
ones. It is to be noted that whenever we have a unitary basis,
taking transposes, adjoints or complex conjugates of every element
we get another unitary basis  (though a `right' basis may give rise
to a `left' basis).  But this is for the same subalgebra system.
Currently we are interested in simple ways of constructing
orthonormal bases for new inclusions. One such method  is
`concatenation', and this  has already been described in Section 2.
We do not know as to whether the constructions described here
exhaust all the subalgebra systems with $U$-property or not.
However, we are able to show that the answer is in the affirmative
if $\mathcal{A}$ or $\mathcal{B}$ is a full matrix algebra.

\subsection{Tensor products and direct sums} Suppose $(\mathcal{B}_i \subseteq  \mathcal{A}_i, E_i)$ for $i=1,2$
are two finite dimensional inclusions with $U$-property. Suppose
$\{U_1(i) , \ldots , U_{d_i}(i)\}$ is a unitary o.n.b. for
$(\mathcal{B}_i, \mathcal{A}_i, E_i)$ for $i=1,2.$ Take
$\mathcal{B}=\mathcal{B}_1\otimes \mathcal{B}_2$. Consider it as a
subalgebra of $\mathcal{A}=\mathcal{A}_1\otimes \mathcal{A}_2$ in
the natural way with the inclusion matrix being the tensor product
of inclusion matrices. Then it is seen easily that $E=E_1\otimes
E_2$ is a conditional expectation map from $\mathcal{A}$ to
$\mathcal{B}.$ For this inclusion, $\mathcal{B}_1\otimes
\mathcal{B}_2\subseteq \mathcal{A}_1\otimes \mathcal{A}_2,
E_1\otimes E_2)$, we can observe that $\{ U_j(1)\otimes U_k(2):
0\leq j\leq d_1; 0\leq k\leq d_2\}$ is a unitary orthonormal basis.
Under the same set up, if $d_1=d_2$, the subalgebra system
$(\mathcal{B}_1\oplus \mathcal{B}_2\subseteq \mathcal{A}_1\oplus
\mathcal{A}_2, E_1\oplus E_2)$ has $U$-property. In fact, $\{
U_j(1)\oplus U_j(2): 0\leq j \leq d_1\}$ is a unitary orthonormal
basis.

\subsection{Basic Construction}
Let $(\mathcal{B}\subseteq \mathcal{A}, E)$ be an inclusion of
finite dimensional algebras.  Then by the famous Jones basic
construction we have an inclusion $(\mathcal{A}\subseteq
\mathcal{A}_1, E_1)$ where the inclusion matrix is the transpose of
the original inclusion matrix and $E_1$ is the dual conditional expectation of $E$ (see \cite{W}). Now it is known that if
$(\mathcal{B}\subseteq \mathcal{A}, E)$ has U-property, so does the
$(\mathcal{A}\subseteq \mathcal{A}_1, E_1)$ . Indeed if $\{U_0,
\ldots , U_{d-1}\}$ is a unitary orthonormal basis for
$(\mathcal{B}\subseteq \mathcal{A}, E)$,  and $J$ is the Jones
projection in $\mathcal{A}_1$, then $\{ W_j:0\leq j\leq (d-1)\}$
where
$$W_j= \sum _{k=0}^{d-1}\epsilon(\frac{jk}{d})U_kJU_k^*$$
is a unitary orthonormal basis for $(\mathcal{A}\subseteq
\mathcal{A}_1, E_1)$, as shown in \cite{S}.

We may employ any of these methods repeatedly and get various
subalgebra systems admitting unitary basis. We also make use of the
trivial fact that for any algebra $\mathcal{A}$, the inclusion
system $(\mathcal{A}\subseteq \mathcal{A}, E)$, where $E$ is the
identity map, admits a unitary basis, namely $\{I\}.$ This allows us
to get the following results.

\begin{thm} Let $(\mathcal{B}\subseteq \mathcal{A}, E)$ be a subalgebra system with $\mathcal{B}=M_{m}$, and  $E$ preserving the Markov trace. Then
$(\mathcal{B}\subseteq \mathcal{A}, E)$ has $U$-property.
\end{thm}

\begin{proof}
As usual take $\mathcal{A}= \oplus _{i=0}^{s-1}M_{n_i}.$ Now
$\mathcal{B}\subseteq \mathcal{A}$ (unitally) means that each $n_i$
is a multiple of $m$, say $n_i=m\times k_i$ for some natural number
$k_i$ and the inclusion matrix is
$$\left[\begin{array}{c}
k_0\\ k_1\\ \vdots \\ k _{s-1}\end{array}\right].$$ Then the
inclusion is isomorphic to the tensor product of $(M_{m}\subseteq
M_{m}, E_1)$ and $(\mathbb{C} \subseteq \oplus _{i=0}^{s-1}M_{k_i},
E_2)$, where $E_1$ is the identity map and $E_2$ is the state:
$$E_2(\oplus _{i=0}^{s-1}Y_i)= \frac{1}{\sum _ik_i^2}\sum
_i~\mbox{trace} (Y_i), ~~ \oplus _i Y_i\in \oplus _i M_{k_i}.$$
Since both of these admit unitary basis (the first one trivially,
and the second one because of Section 4), $(\mathcal{B}\subseteq
\mathcal{A}, E)$ also admits a unitary basis.

\end{proof}

\begin{thm}
Let $(\mathcal{B}\subseteq \mathcal{A}, E)$ be a subalgebra system
with $\mathcal{A}=M_{n}$ and $E$ preserving the unique trace on
$M_{n}.$ Then $(\mathcal{B}\subseteq \mathcal{A}, E)$ has
$U$-property.
\end{thm}

\begin{proof}
Suppose $\mathcal{B}=\oplus _{j=0}^{r-1}M_{m_j}.$ Now the inclusion
matrix is a row vector, say $A=[a_0, a_1, \ldots, a_{r-1}]$
satisfying $n=\sum _j a_jm_j.$ Now the spectral condition implies
that there exists $d\in \mathbb{N}$ such that
$$m_j.d=na_j, ~\forall j.$$
So $a_j=\frac{d}{n}.m_j, ~~\forall j.$ Let $k,l$ be relatively prime
natural numbers such that $\frac{l}{k}=\frac{d}{n}.$ As $a_j$ is an
integer, $k$ divides $m_j$. Take $\tilde{m}_j = \frac{m_j}{k}.$ Then
$$n= \sum _ja_jm_j=\sum _j\frac{d}{n}m_j^2=kl\sum _j\tilde{m}_j^2$$

Consider three inclusions $\mathcal{B}_0=M_k$, $\mathcal{A}_0=M_k$,
$\mathcal{B}_1=\mathbb{C}, \mathcal{A}_1= M_l$ and $
 \mathcal{B}_2=\oplus _j
M_{\tilde{m}_j}$ and $ \mathcal{A}_2=M_{\sum _j\tilde{m}_j^2}.$ We
observe that $(\mathcal{B}_1, \mathcal{A}_1, E_1)$ where $E_1$ is
the normalized trace has U-property. Also $\mathbb{C}\subseteq
\mathcal{B}_2$ (with appropriate conditional expectation) has
$U$-property,  by  \cite{CKP}. On employing Jones basic construction
to this inclusion, $\mathcal{B}_2\subseteq \mathcal{A}_2$ has
$U$-property. Taking the tensor product of three inclusions we see
that $\mathcal{B}\subseteq \mathcal{A}$ with its conditional
expectation preserving the Markov state has $U$-property.
\end{proof}

\section{Applications}
We provide two applications of the theory of unitary orthonormal
bases we developed in the previous sections. One is in subfactor
theory and the other one is in Connes-St{\o}rmer relative entropy.
\subsection{Depth 2 subfactor and unitary
Pimsner-Popa basis} In this subsection we focus on a subfactor $N\subset M$ of type $II_1 $ factors with $[M:N]<\infty$. There exists a unique trace preserving conditional expectation $E$ from $M$ onto $N$. Motivated by the finite dimensional case, we call the triple $(N\subset M,E)$ a \textit{subfactor system}. Popa asked whether there exists a unitary orthonormal basis for an irreducible subfactor $N\subset M$ with $[M:N]$ being an integer. See \cite{pop2} for various motivations of this open problem. The following modification of the above problem also seems to be very interesting (see \cite{BG,BG2}).
\begin{qn}
Does there exist a unitary orthonormal basis for an integer index extremal subfactor?
\end{qn}
Recently, in \cite{BG} it has been proved that any finite index
regular subfactor with either simple or abelian relative commutant
must have unitary orthonormal basis. Also, in \cite{BG} it was
conjectured that any  regular subfactor will have unitary
orthonormal basis. Indeed, in \cite{CKP} this conjecture has been
verified by proving that any regular subfactor with finite Jones
index has unitary orthonormal basis. As pointed out in \cite{BG} and
\cite{CKP}, as a cute application of the existence of the unitary
orthonormal basis one can conclude that any finite index regular
subfactor has depth at most $2$. The existence of unitary
orthonormal basis for a general integer index and depth 2 subfactor
seems to be unknown. In \cite{BG2} we managed to prove that any
finite index and depth $2$ subfactor $N\subset M$ with
$N^{\prime}\cap M$ simple must have unitary orthonormal basis. In
this paper, we have a complementary result as in \Cref{revisit}.
Before stating the result, let us recall few important definitions. \smallskip

Consider an inclusion of finite dimensional $C^*$-algebras
$\mathcal{N\subset M}$ with the Bratelli diagram connected and a
conditional expectation $E:\mathcal{M}\rightarrow \mathcal{N}$. It
follows that the dual inclusion ${\mathcal{M}}^{\prime}\subset
{\mathcal{N}}^{\prime}$ also has connected Bratelli diagram and
there exists a dual conditional expectation $E^{\prime}:
{\mathcal{N}}^{\prime}\rightarrow {\mathcal{M}}^{\prime}.$ Then, by
Theorem 2.6 in \cite{L}, there exists a tracial state $\omega_s$ on
${\mathcal N}^{\prime}\cap \mathcal{M}$ which is the pointwise
strong limit of the words $\{EE^{\prime} EE^{\prime}\cdots ,
E^{\prime} EE^{\prime} E\cdots \}$.
\begin{defn}\cite{L}
(\textit{Minimal conditional expectation}) Suppose we have an
inclusion of finite dimensional von Neumann algebras
$\mathcal{N\subset M}$ with the Bratelli diagram connected. A
conditional expectation $E:\mathcal{M}\rightarrow \mathcal{N}$ is
called \textbf{minimal} if the norm of the Watatani index $\lVert
\text{ind}_w E\rVert $ is minimal among all other conditional
expectations from $\mathcal{M}$ onto $\mathcal{N}$. This minimal
value is called the {minimal index} of $\mathcal{N\subset M}.$
 \end{defn}
 \begin{defn} \cite{L}
 (\textit{Superextremal inclusion})
Consider an inclusion of finite dimensional $C^*$-algebras
$\mathcal{N\subset M}$ with the connected Bratelli diagram. We
denote the Markov trace by $\tau$. The pair $\mathcal{N\subset M}$
is said to be a \textbf{superextremal inclusion} if the
$\tau$-preserving conditional expectation is the minimal conditional
expectation and $\omega_s= \tau |_{{\mathcal{N}}^{\prime}\cap
\mathcal{M}}.$
\end{defn}
\begin{defn} (\textit{Finite depth subfactor})
 Consider a finite index inclusion $N\subset M$ of $II_1$-factors and
 suppose $N\subset M\subset M_1\subset \cdots \subset M_k\subset
 \cdots $ is the corresponding tower of Jones' basic construction. Then, the
 inclusion $N\subset M$ is said to have \textit{finite depth} if there
 exists a $k$ such that $N^{\prime}\cap M_{k-2}\subset N^{\prime}\cap
 M_{k-1}\subset N^{\prime}\cap M_{k}$ is an instance of basic
 construction.  The least such $k$ is defined as the \textit{depth} of
 the inclusion.
 \end{defn}
The study of depth 2 subfactor has attracted a good deal of attention. A  finite depth subfactor is an intermediate subfactor of a depth 2 subfactor.
It is also a well-known fact that any depth 2 subfactor can be characterized by an action of an weak Hopf algebra.
\begin{thm}\label{revisit}
Consider a subfactor $N\subset M$ with finite Jones index and
suppose $E$ denotes the unique $tr$-preserving conditional
expectation. If $N\subset M$ is a depth 2 subfactor with
$N^{\prime}\cap M$ abelian and $N^{\prime}\cap M\subset
N^{\prime}\cap M_1$ is superextremal then the subfactor system
$(N\subset M, E)$  admits a unitary orthonormal basis.
 \end{thm}
 \begin{proof}
We denote the Markov trace for $M\subset M_1$ by $tr_1$ which is the
extension of the trace on $M$. We also denote by $E_{N^{\prime}\cap
M}$ the $tr_1|_{N^{\prime}\cap M_1}$-preserving conditional
expectation from $N^{\prime}\cap M_1$ onto $N^{\prime}\cap M$.
 As $N\subset M$ is of depth 2, we have the following non-degenerate commuting square with respect to $tr_1$ \[
\begin{array}{ccc}
M & \subset & M_1\\
\cup & & \cup\\
N'\cap M & \subset & N'\cap M_1\end{array}.
\] In other words, $E_M E_{N^{\prime}\cap M}= E_{N^{\prime}\cap M} E_M $
and any Pimsner-Popa basis for $E_{N^{\prime}\cap M}$ is also a Pimsner-Popa basis for $E_M$.
Since, $N^{\prime}\cap M\subset N^{\prime}\cap M_1$ is superextremal, by Proposition 4.4 in \cite{L}, we see that $N^{\prime}\cap M\subset N^{\prime}\cap M_1$ satisfies the spectral condition as in \Cref{spectral}. By \Cref{abelain unitary} we conclude that we have a unitary orthonormal basis for $(N^{\prime}\cap M\subset N^{\prime}\cap M_1, E_{N^{\prime}\cap M})$ and so, for $(M\subset M_1,E_M)$ too. The proof is complete once we recall that $N\subset M$ is of depth 2 iff $N_{-1}\subset N$ is also of depth 2, where $N_{-1}\subset N$ is a model of the downward basic construction of $N\subset M$. \end{proof}

\begin{rem}
From \Cref{revisit} we conclude if a depth 2 subfactor $N\subset M$ has  the property that  $N^{\prime}\cap M$ is abelian and $N^{\prime}\cap M\subset
N^{\prime}\cap M_1$ is superextremal, then  $[M:N]$ must be an integer. At present, we don't know whether the superextremality condition in \Cref{revisit} will be automatically satisfied for a depth 2 subfactor or not.
 \end{rem}
\smallskip

\subsection{Quantum relative Entropy and unitary orthonormal basis} Generalizing the classical notion of the conditional entropy from ergodic theory,
Connes and St{\o}rmer in \cite{CS} defined a relative entropy
$H(\mathcal{P|Q})$ between a pair of finite dimensional von
Neumann-subalgebras $\mathcal{P}$ and $\mathcal{Q}$ of a finite von
Neumann algebra $\mathcal{M}$ with a fixed faithful normal trace.
Using the relative entropy as the main technical tool they succeeded
to prove a non-commutative version of the Kolmogorov-Sinai type
theorem. Subsequently, Pimsner and Popa (in \cite{PP})  observed
that the definition of the Connes-St{\o}rmer relative entropy does
not depend on $\mathcal{P,Q}$ being finite dimensional and we may
consider the relative entropy $H(\mathcal{P|Q})$ for arbitrary von
Neumann subalgebras $\mathcal{P,Q \subset M}.$ The most striking
result they discovered is the close relationship between the Jones
index of a type $II_1$- subfactor $N\subset M$ and the relative
entropy $H(M|N)$. We refer the reader to \cite{NS} for a
comprehensive study of  relative entropy.

The relative entropy for an inclusion of finite dimensional von
Neumann algebras  happens  to be a fundamental tool in subfactor
theory (see \cite{Po3,Po4}). For an inclusion of multi-matrix
algebras $(\mathcal{B}\subseteq \mathcal{A}, E)$,  with $E$
preserving a faithful trace, Pimsner and Popa in \cite{PP} have
provided very useful formulae for the Probabilistic index
$\lambda(\mathcal{B\subset A},E)$ and the relative entropy
$H(\mathcal{A|B})$. Using those formulas, in \cite{PP2}, Pimsner and
Popa investigated the relationship between the relative entropy $H$
and the norm of the inclusion matrix $A$. One can observe from
\cite{PP2} that the equality of $H(\mathcal{A} |\mathcal{B})$ and
$\ln {\lVert A \rVert}^2$ is a delicate matter. In this paper,
we revisit this problem.
  \begin{thm}\label{entro}
Consider an inclusion of finite dimensional von Neumann algebras
$\mathcal{B\subset A}$ with an inclusion matrix $A$ and suppose
$\mathcal{A}_1$ is the basic construction corresponding to the
Markov trace $\tau$. If the unique $\tau$-preserving conditional
expectation has a unitary orthonormal basis, then
$$H(\mathcal{A}_1|\mathcal{A})=\ln {\lVert A \rVert}^2.$$
\end{thm}
In order to prove the above theorem we first recall a result by Pimsner and Popa.
\begin{prop}\cite{PP}  \label{pimsnerpopa} Let $\mathcal{N\subset M}$ be arbitrary finite von Neumann algebras with a trace $tr$ and $q\in \mathcal{M}$ a projection such that $E^{\mathcal{M}}_{\mathcal{N}^{\prime}\cap \mathcal{M}}(q)=cf $ for some scalar $c$ and some projection $f\in N^{\prime}\cap M$. Then,
$$H(\mathcal{M|N}) \geq c^{-1} tr\big( \eta E_{\mathcal{N}}(q)\big)$$
\end{prop}

\medskip

\textit{Proof of  \Cref{entro}} :  Suppose $tr_{\mathcal{A}}$ is the Markov trace for the  inclusion $\mathcal{B\subset A}$. It  is well-known that the minimal conditional expectation $E_0: \mathcal{A\rightarrow B}$ is precisely the unique $tr_{\mathcal{A}}$-preserving conditional expectation. Fix a unitary Pimsner-Popa basis $\{u_i:i\in I\}$ for $(\mathcal{B,A},E_0)$. The trace $tr_{\mathcal{A}}$ extends to a trace on $\mathcal{A}_1$, denoted by $tr_{\mathcal{A}_1}$, which is also the Markov trace for the inclusion $\mathcal{A}\subset \mathcal{A}_1$. Define a trace $tr_{{\mathcal{B}}^{\prime}}$ on the von Neumann algebra ${\mathcal{B}}^{\prime}$ by  $tr_{{\mathcal{B}}^{\prime}}(x^{\prime})=tr_{\mathcal{A}_1}(x_1),$ where $x^{\prime}=J_{\mathcal{A}}x_1J_{\mathcal{A}}$ for some $x_1\in \mathcal{A}_1$. It is easy to check that $tr_{{\mathcal{B}}^{\prime}}$ is a trace and moreover, it is the Markov trace for the inclusion ${\mathcal{A}}^{\prime}\subset {\mathcal{B}}^{\prime}$.  Following \cite{B}[Proposition 2.7] (see also \cite{JP}[Proposition 2.24]),  the unique $tr_{{\mathcal{B}}^{\prime}}$-preserving conditional expectation $E^{{\mathcal{B}}^{\prime}}_{{\mathcal{A}}^{\prime}}$ is given by
\begin{equation}\label{condi}
E^{{\mathcal{B}}^{\prime}}_{{\mathcal{A}}^{\prime}}(x^{\prime})={\lVert A\rVert}^{-2} \sum_i u_i x^{\prime}u^*_i.
\end{equation}
Now if $E_0$ has unitary orthonormal basis we note that the dual conditional expectation $E_1:{\mathcal{A}}_1\rightarrow \mathcal{A}$ will also have a unitary orthonormal basis and therefore, it satisfies the spectral condition as in \Cref{spectral} and  hence, it is superextremal (see \cite{L}[Proposition 4.4]) and therefore, $tr_{{\mathcal{A}}_1}|_{{\mathcal{A}}^{\prime}\cap A_1}= tr_{{\mathcal{A}}^{\prime}}|_{{\mathcal{A}}^{\prime}\cap {\mathcal{A}}_1}= tr_{{\mathcal{B}}^{\prime}}|_{A^{\prime}\cap A_1}.$ It follows that if $e_1$ deontes the Jones projection corresponding to $E_0$ then we must have:
\begin{equation}\label{imp}
E^{{\mathcal{A}}_1}_{{\mathcal{A}}^{\prime}\cap {\mathcal{A}}_1}(e_1)= E^{{\mathcal{B}}^{\prime}\cap {\mathcal{A}}_1}_{{\mathcal{A}}^{\prime}\cap {\mathcal{A}}_1}(e_1).
\end{equation} Since, $\{u_i\}$ is a basis we have $\sum_iu_ie_1u^*_i=1$ (see \cite{Bak}, for instance) and therefore, by \Cref{condi} we get $E^{{\mathcal{B}}^{\prime}\cap {\mathcal{A}}_1}_{{\mathcal{A}}^{\prime}\cap {\mathcal{A}}_1}(e_1)= {\lVert A \rVert}^{-2} $ and by \Cref{imp}, $E^{{\mathcal{A}}_1}_{{\mathcal{A}}^{\prime}\cap {\mathcal{A}}_1}(e_1)= {\lVert A \rVert}^{-2}.$
Using \Cref{pimsnerpopa} we get $H(\mathcal{A}_1|\mathcal{A})\geq \ln{\lVert A \rVert}^{2}.$ By Proposition 2.6.3 of \cite{W}, we see that $\lambda({\mathcal{A}}_1, \mathcal{A}) ={\lVert A \rVert}^{-2}.$ Thanks to Proposition 3.5 of \cite{PP}, we obtain $H({\mathcal{A}}_1|\mathcal{A})\leq -\ln \lambda({\mathcal{A}}_1, \mathcal{A})=\ln {\lVert A \rVert}^2.$ Thus, we conclude that $H({\mathcal{A}}_1| \mathcal{A})=\ln {\lVert A \rVert}^{-2}.$ \qed

\subsection*{Acknowledgements}
 This collaboration got initiated at the {\em Discussion meeting on
Linear Analysis\/} held at Evolve Back, Coorg, during Feb. 19-23,
2023. We thank the Indian Academy of Sciences, Bangalore for funding
this program. A part of this work was done when the first  author was visiting the second at Indian Statistical Institute, Bangalore. Bakshi thanks ISI Bangalore for kind hospitality.  Bakshi was supported through a DST INSPIRE faculty
grant  No. DST/INSPIRE/04/2019/002754. Bhat gratefully acknowledges
funding from  SERB, Anusandhan National Research Foundation (India)
through J C Bose Fellowship No. JBR/2021/000024.

\end{document}